\documentclass{amsart}

\usepackage{amsmath}
\usepackage{amssymb}
\usepackage{amscd}
\usepackage{hyperref}

\newtheorem{theorem}{Theorem}[section]
\newtheorem{thm}[theorem]{Theorem}
\newtheorem{cor}[theorem]{Corollary}
\newtheorem{lem}[theorem]{Lemma}
\newtheorem{prop}[theorem]{Proposition}

\theoremstyle{definition}
\newtheorem{defn}[theorem]{Definition}

\newtheorem{rem}[theorem]{Remark}
\newtheorem{constr}[theorem]{Construction}

\newtheorem{conj}[theorem]{Conjecture}

\theoremstyle{remark}

\newcommand{\mbb}{\mathbb}
\newcommand{\QQ}{\mbb{Q}}

\newcommand{\CC}{\mbb{C}}
\newcommand{\RR}{\mbb{R}}

\newcommand{\PP}{\mbb{P}}

\newcommand{\FF}{\mbb{F}}

\newcommand{\mc}{\mathcal}

\newcommand{\OO}{\mc{O}}

\newsavebox{\sembox}
\newlength{\semwidth}
\newlength{\boxwidth}

\newsavebox{\semrbox}
\newlength{\semrwidth}
\newlength{\boxrwidth}

\newcommand{\GWpp}{\langle [pt], [pt], \ldots\rangle_{0, \beta}^{X}}

\title
{Symplectic geometry and rationally connected $4$-folds}

\author[Tian]{Zhiyu Tian}
\address{
Department of Mathematics 253-37\\
California Institute of Technology \\ 
Pasadena, CA, 91125}
\email{tian@caltech.edu}

\date{\today}

\begin{document}


\begin{abstract}
We study some symplectic geometric aspects of rationally connected $4$-folds. As a corollary, we prove that any smooth projective $4$-fold symplectic deformation equivalent to a Fano $4$-fold of pseudo-index at least $2$ or a rationally connected $4$-fold whose second Betti number is $2$ is rationally connected. 
\end{abstract}


\maketitle


\section{Introduction}
This paper is devoted to the study of some symplectic geometric aspects of rationally connected $4$-folds. The key motivation is the following conjecture.

\begin{conj}[Koll{\'a}r, \cite{KollarUni}]\label{conj:Kollar}
Let $X$ and $Y$ be two smooth projective varieties which are symplectic deformation equivalent. Then $X$ is rationally connected if and only if $Y$ is.
\end{conj}

Here two symplectic manifolds are symplectic deformation equivalent if there is a family of smooth symplectic manifolds $\{(V_t, \omega_t), 0 \leq t \leq 1\}$ that connects the two given ones. This conjecture is known to be true in the second Betti number $1$ case (\cite{KollarUni}), in dimension $1, 2$ (which follow easily from classification), and $3$ (\cite{VoisinRC}, \cite{SymRC}). 

One of the main results of this paper is the following.

\begin{thm}\label{thm:SymRC}
Let $Y$ be a smooth projective fourfold. Assume $Y$ is symplectic deformation equivalent a smooth $4$-fold $X$ which is either
\begin{enumerate}
\item a Fano $4$-fold of pseudo-index at least $2$, or
\item a rationally connected $4$-fold whose second Betti number is $2$.
\end{enumerate}
Then $Y$ is rationally connected.
\end{thm}

Recall that the pseudo-index $i(X)$ of a smooth projective Fano variety is defined to be the minimum of the set $\{-K_X \cdot C \vert ~C \text{~is a rational curve in}~X.\}$.

The two cases are treated differently. First of all, the Fano case follows from the following theorem.

\begin{thm}\label{thm:main}
Let $X$ be a smooth projective Fano $4$-fold of pseudo-index at least $2$. Then there is a non-zero Gromov-Witten invariant of the form $\langle [pt], [pt], \ldots \rangle^X_{0, \beta}$.
\end{thm}

Indeed, Gromov-Witten invariants are symplectic deformation invariants. Thus by Theorem \ref{thm:main} there is a non-zero Gromov-Witten invariant of the form $\langle [pt], [pt], \ldots \rangle^Y_{0, \beta}$ on $Y$. Then the moduli space of rational curves containing $2$ points is non-empty otherwise the Gromov-Witten invariant is $0$. Therefore $Y$ is rationally connected.

The proof of Theorem \ref{thm:main} naturally divides into two different cases. When the Fano variety has Picard number $1$, then the quotient construction of Koll{\'a}r-Miyaoka-Mori (\cite{KMM92RCFano}, \cite{KMM92RC}) gives a low degree very free rational curve (i.e. a rational curve on which the pull back of the tangent bundle of $X$ is ample). Then an easy argument involving bend-and-break allows one to understand the boundary of the moduli space. When the variety has higher Picard number, we use Mori contractions to understand the geometry of the variety. Then we could either reduce this problem to lower dimensional cases or construct low degree very free curves. 

Our technique also gives the following theorem.
\begin{thm}[=Theorem \ref{thm:BalancedCurve}]
Let $Y$ be a Fano $4$-fold and $f: C \to Y$ be an embedded rational curve. Assume that the normal bundle of $C$ is isomorphic to $\OO(2) \oplus \OO(2) \oplus \OO(2)$. Then there is a non-zero Gromov-Witten invariant of the form $\langle [pt], [pt], [pt] \rangle^Y_{0, C}$ or $\langle [pt], [pt] \rangle^Y_{0, \tilde{C}}$.
\end{thm}

\begin{rem}\label{rem:balance}
Let $X$ be a Fano $4$-fold of Picard number $1$ and $C$ be a free curve of minimal degree. Thus $-K_X \cdot C$ is $2, 3, 4$ or $5$. Then the quotient construction of Koll{\'a}r-Miyaoka-Mori gives an embedded very free curve of $-K_X$ degree $8$ except when $-K_X \cdot C=3, 5$. The question is whether or not the normal bundle of a general such curve is balanced. Given the results in \cite{ShenMM}, it is reasonable to expect that they are actually balanced. It is also easy to see that when $-K_X \cdot C=4$, the normal bundle of the very free curve is indeed balanced. If $-K_X \cdot C=5$, the curve $C$ is very free with balanced normal bundle $\OO(1) \oplus \OO(1) \oplus \OO(1)$.
\end{rem}

The $b_2(X)=2$ case is different in that we are unable to prove the existence of a non-zero Gromov-Witten invariant with two point insertions. Instead we prove a more precise result about uniruled $4$-fold whose second Betti number is $2$. 

\begin{thm}\label{thm:KollarConj}
Let $W$ and $W'$ be two smooth projective uniruled $4$-fold which are symplectic deformation equivalent and has second Betti number $2$. Then the MRC fibration of $W$ and $W'$ has the same dimension and the class of a general fiber has the same cohomology class.
\end{thm}

The idea of the proof is motivated by symplectic birational geometry (c.f. \cite{HLR}) which is the analogue of birational geometry in the symplectic category. This program in particular predicts that symplectic deformation equivalent varieties should have similar behaviors in birational geometry. One interesting question is whether the minimal model program (MMP) for a variety is something symplectic. Perhaps the first step is to study the behavior of extremal rays under symplectic deformations. 

\begin{thm}\label{thm:extremal}
Let $R$ be an extremal ray of a smooth projective $4$-fold $X$. Assume one of the followings:
\begin{enumerate}
\item The extremal contraction does not contract a divisor to a point.
\item The contraction contracts a divisor to a point. And the divisor is isomorphic to $\PP^3$ with normal bundle $\OO_{\PP^3}(-a), a=1, 2, 3$, or a smooth quadric hypersurface $Q$ in $\PP^4$ with normal bundle $\OO_{Q}(-a), a=1, 2$, or a del Pezzo threefold with at worst rational singularities or isolated singularites.
\end{enumerate} Then there is a curve class $[\alpha]$, spanning $R$, such that there is a non-zero Gromov-Witten invariant $\langle A_1, \ldots, A_n \rangle^X_{0, [\alpha]}$.
\end{thm}

Previous results for surfaces and threefolds are obtained by Ruan \cite{RuanExtremal}. Indeed, Ruan observed that this type of results has many interesting consequences regarding the symplectic topology of the underlying symplectic manifolds. In the following, we denote the corresponding symplectic manifold of a smooth projective variety by $(V, \omega)$. First recall some definitions from~\cite{RuanExtremal}.
\begin{defn}
Let $J$ be an almost complex structure on $V$, compatible with $\omega$. Define the $J$-effective cone $NEJ(V)$ to be the subset of
$H_2(V, \RR)$ given by

\{$\sum a_i[C_i] \vert C_i$ is a $J$-holomorphic curve, $a_i \in \RR ~\text{and } a_i\geq 0 $.\}
\end{defn}

\begin{defn}
Define the symplectic effective cone $SNE(V) = \cap NEJ(V)$ for all compatible $J$.

Define the deformed symplectic effective cone $DNE(V) = \cap SNE(V)$ for all the symplectic forms which can be deformed to $\omega$.
\end{defn}
Since Gromov-Witten invariants are symplectic deformation invariant and are zero if the moduli space of $J$-holomorphic curves is empty for some almost complex structure $J$, we get the following corollary.
\begin{cor}\label{cor:extremal}
Extremal rays listed in Theorem \ref{thm:extremal} on smooth projective $4$-folds are
symplectic and deformed symplectic extremal.
\end{cor}

This corollary, in particular, shows that all such extremal rays lie in the \emph{effective} cone of any variety symplectic deformation equivalent to the given variety. Note that this is the best one can hope for since the surface $\PP^1 \times \PP^1$ will deform to the Hirzebruch surface $\FF_2$ and only one of the two extremal rays stay extremal (but both stay effective). We expect that all the extremal rays stay effective under symplectic deformations. Such a statement will follow if one can show all the extremal rays support some curve class which gives a non-zero Gromov-Witten invariant. It seems that there is no way to prove such a statememnt without an explicit classification result and computations. However, we note the following.

\begin{thm}\label{thm:ray}
Let $Z$ and $Z'$ be two symplectic deformation equivalent varieties with second Betti number $2$. Assume that there is a cohomology class $A \in H^2(Z, \QQ)=H^2(Z', \QQ)$ which is an ample class for both $Z$ and $Z'$. Then any extremal ray of $Z$ (resp. $Z'$) lie in the effective cone of $Z'$ (resp. $Z$) and at least one of the extremal rays stays extremal.
\end{thm}

There are several situations where the assumptions are satisfied, e.g. when $(Z, \omega)$ and $(Z', \omega')$ are symplectomorphic, or symplectic isotopy to each other (and the combination of the two). The idea of the proof comes from a result of Wi{\'s}niewski \cite{DefNefValue}. In fact, all one need to do is to adapt the proof to our case.

It is easy to see that if all the extremal rays stay effective under symplectic deformations, then any variety $W$ symplectic deformation equivalent to such a $Z$ (assuming $b_2(Z)=2$) has the same extremal ray under the identification $H_2(Z, \QQ) \cong H_2(W, \QQ)$. Hence the minimal model program (MMP) will produce the same fibration structure (at least in dimension $4$, c.f. Lemma \ref{lem:extremalray}). Then a symplectic Graber-Harris-Starr type result \cite{SymRC} will prove Koll{\'a}r's conjecture in this special case. However, due to the fact that we are unable to compute all the Gromov-Witten invariants associated to the extremal rays, we need, in addition, a topological argument to show the fibration structure is preserved under symplectic deformation.

\textbf{Acknowledgement:} I would like to thank Tom Graber, Yongbin Ruan, Jason Starr, Claire Voisin, and Aleksey Zinger for discussions about various part of the project, Zhiyuan Li and Letao Zhang for their encouragement, and Chenyang Xu for providing the key argument in Lemma \ref{lem:structure} using ``fake projective spaces".

\section{rational curves of small $-K_X$ degree}\label{sec:2}
The main theorem of this section is the following.
\begin{thm}\label{thm:lowdegree}
Let $X$ be a smooth projective Fano $4$-fold of pseudo-index at least $2$ and $f: C\cong \PP^1 \to X$ be a very free curve such that $-f^*K_X \cdot C$ is at most $8$. Then there is a non-zero Gromov-Witten invariant with two point insertions that counts the number of irreducible very free curves satifying the incidence relations.
\end{thm}

The strategy of the proof is the following. We want to show that any curve (in the chosen class) that can meet all the constraints are irreducible free curves, which in turn implies that it contributes positively to the Gromov-Witten invariant. To do this, we begin with curves meeting two very general points. If this curve is reducible and meet other constraints, then we show that it is smoothable. Also it will lie in a component whose general point parametrize a very free rational curve. In particular, this component has the expected dimension. Then we can choose the constraints such that only general curves (hence irreducible) in the components of expected dimension can meet all the constraints.

We first give $2$ sufficient conditions on when a reducible curve is smoothable.
\begin{lem}\label{lem:smoothing}
Let $X$ be a smooth projective variety. Let $C=C_1 \cup C_2 \cup C_3$ be a chain of $\PP^1$s and $f: C \to X$ be a stable map. Assume that $C_1$ and $C_2$ are free curves and the Hom-scheme $Hom(\PP^1, X)$ has expected dimension $\dim X+(-K_X \cdot C_2)$ at the point $(f: C_2 \to X)$. Then the general point of any irreducible component of the Kontsevich moduli space containing $(C, f)$ parametrizes to an irreducible curve. If one of the free curves contains a very general point, then a very general deformation is free.
\end{lem}

\begin{proof}
This follows from the estimate in Theorem 7.9 in Chapter 2 of \cite{Kollar96}. Namely, for a comb with $m$-teeth, there is a subcomb with $m'$ teeth that is smoothable. And a lower bound of $m'$ is given by 
\begin{equation}\label{eq:1}
m' \geq m-(K_X \cdot C_2)+\dim X-\dim_{[f\vert_{C_2}]}Hom(\PP^1, X).
\end{equation}
By assumption, $-(K_X \cdot C_2)+\dim X=\dim_{[f\vert_{C_2}]}Hom(\PP^1, X)$. Thus $m'=2$ and the reducible curve is smoothable. If one of the free curves contains a very general point, a very general deformation of this reducible curve will also contain a very general point, thus free.
\end{proof}

\begin{lem}\label{lem:smoothable}
Let $X$ be a Fano $4$-fold of pseudo-index $i(X)$ at least $2$. Let $C=C_1 \cup C_2 \cup C_3$ be a chain of $\PP^1$s and $f: C \to X$ be a stable map. Assume that $-f^*K_X \cdot C_2\leq 3$ and $C_i$ passes through a very general point for $i=1, 3$. Then a general point of any irreducible component of the Kontsevich moduli space containing $(C, f)$ corresponds to an irreducible very free curve.
\end{lem}

\begin{proof}
We first show that every irreducible component of $Hom(\PP^1, X)$ at $[C_2]$ has dimension at most $7$. If the deformation of $C_2$ in one component covers the whole $X$, a general point of that irreducible component parametrizes a free rational curve. Thus the dimension is exactly $-K_X \cdot C_2+\dim X \leq 7$. So for now assume that the deformation of $C_2$ in one irreducible component $U$ is contained in a subvariety $S$. Both $C_1$ and $C_3$ intersect the subvariety $S$ at finitely many points and $C_2$ has to pass through at least two of them. Then $C_2$ does not move once we make the choice of the two points. Otherwise we can deform $C_2$ fixing two points and, by bend-and-break, $C_2$ breaks into a reducible or non-reduced curve. But this cannot happen since $-K_X \cdot C_2\leq 3$, $i(X) \geq 2$ and $-K_X$ is ample. The above analysis shows that the evaluation map
\[
\PP^1 \times \PP^1 \times Hom(C_2, X) \to X\times X
\] 
has fiber dimension $3$ (coming from reparametrization ) and the image has dimension $\dim S \times 2\leq 3 \times2=6$. Therefore,
\[
\dim_{[f\vert_{C_2}]}Hom(C_2, X) \leq 7.
\]

The estimate \ref{eq:1} in Lemma \ref{lem:smoothing} shows that at least one of the nodes can be smoothed out. If only one node is smoothed, then we get a reducible curve, which is the union of two irreducible curves each passing through a very general point. Then a general deformation of this new curve is irreducible and very free. 
\end{proof}

We need to analyze the normal bundle of the minimal very free curve. The following construction, used in \cite{ShenMM}, is very helpful.

\begin{constr}
Let $C$ be a very free curve with $-K_X$ degree equal to either $7$ or $8$ and general in an irreducible component of moduli space of very free curve.  Assume the normal bundle of $C$ is $\OO(1)\oplus \OO(1)\oplus\OO(a), a=3 ~\text{or}~4$. 

 Choose $2$ points in $C$ and deform $C$ with the two points fixed. Then the deformation of $C$ sweeps out a surface $\Sigma$ in $Y$. Let $\Sigma'$ be the normalization and $\tilde{\Sigma}$ be the minimal resolution of $\Sigma'$. 
\end{constr}

We need the following definition.

\begin{defn}
Let $C_i \subset X_i$ be a curve on a variety $X_i$, $i=1, 2$. We say $(X_1,C_1)$ is \emph{equivalent to} $(X_2, C_2)$ if there is an open neighborhood $V_i$ of $C_i$ in $X_i$ and an isomorphism $f: V_1 \rightarrow V_2$ such that $f\vert_{C_1}:C_1 \rightarrow C_2$ is also an isomorphism.
\end{defn}

The following results are proved in \cite{ShenMM}, Section 2.2, 2.3.

\begin{prop}[\cite{ShenMM}, Corollary 2.2.7, Proposition 2.3.3, Lemma 2.3.13]\label{prop:Shen}
Notations as above.

\begin{enumerate}
\item $\Sigma$ is independent of the choice of the points. ${\Sigma}'$ is smooth along $C$ and $N_{C/\Sigma'}\cong \OO(a)$. 
\item There is a neighborhood $U$ of $C$ such that the map $\phi : \tilde{\Sigma} \rightarrow X$ has injective tangent map.
 And the normal sheaf $N_{\tilde{\Sigma}/X}$ is locally free along $C$ and $N_{\tilde{\Sigma}/X}\vert_C \cong \OO(1)\oplus \OO(1)$.
\item The pair $(\tilde{\Sigma}, C)$ is equivalent to $(\PP^2, \text{conic})$ or $(\FF_n, \sigma)$, where $\FF_n$ is the $n$-th Hirzebruch surface and $\sigma$ is a section of $\FF_n \to \PP^1$. 

\end{enumerate}
\end{prop}

Note that a very free curve in $\Sigma'$ deforms to a very free curve in $X$. Indeed, it suffices to show that for two general points, the very free curve deforms to a curve that contains these two points. But one can first choose the curve $C$ to contain these two points. Then a deformation of the very free curve contains these two points.

\begin{lem}
The pair $(\tilde{\Sigma}, C)$ is not equivalent to $(\PP^2, \text{conic})$.
\end{lem}

\begin{proof}
If the pair $(\tilde{\Sigma}, C)$ is equivalent to $(\PP^2, \text{conic})$, there exists a rational curve $L$ such that the pair $(\tilde{\Sigma}, L)$ is equivalent to the pair $(\PP^2, \text{line})$. The curve $L$ deforms to a very free curve in $X$. But in this case $-K_X \cdot L \leq 4$, which is impossible since a very free curve in a $4$-fold has $-K_X$ degree at least $5$.
\end{proof}

\begin{lem}
If the pair $(\tilde{\Sigma}, C)$ is equivalent to $(\FF_n, \sigma)$, then there is a very free curve with smaller $-K_X$ degree.

\end{lem}
\begin{proof}
Let $D$ be the unique (possibly reducible) section such that $D \cdot D=-n, (n \geq 0)$. Write $C=D+cF$. Then $c \geq n$ since $c-n=C \cdot D \geq 0$. In addition,
\[
a=C \cdot C=2c-n \geq c,
\]
and
\[
2c \geq 2c-n=C \cdot C =a.
\]
On the other hand, we have
\[
2c \leq c(-K_X \cdot F) \leq -K_X \cdot C=a+4.
\]
Note that $-K_X \cdot F \geq 2$ since a general fiber $F$ is a free curve. Thus we have the following possibilities.
\begin{enumerate}
\item $a=3, c=3$. Then $-K_X \cdot F=2$ and $-K_X \cdot D=1$. This is impossible since the pseudo-index of the Fano variety $X$ is assumed to be at least $2$.
\item $a=3, c=2$. Since $-K_X \cdot D \geq 2$, $-K_X \cdot F$ is not $3$. Then $-K_X \cdot F=2$ and $n=1, -K_X \cdot D=3$.  
\item $a=4, c=4$. Then $-K_X \cdot F=2$. But this would imply that there are two free curves in $X$ with $-K_X$ degree $2$, each passing through a very general point, intersecting each other. This is impossible if the two points are chosen to be very general.
\item $a=4, c=3$. Then $-K_X \cdot F=2$ and $n=2$.
\item $a=4, c=2$. Then $-K_X \cdot F=2$, $3$ or $4$ and $n=0$. But in this case the section $D$ is a moving curve. So $-K_X \cdot D \geq 2$. Thus $-K_X \cdot F$ cannot be $4$.
\end{enumerate}
In the cases that has not been ruled out, $c\geq n+1$. Thus there is a very free curve (in the surface $\tilde{\Sigma}$) of the class $D+(c-1)F$ since the pair $(\tilde{\Sigma}, C)$ is equivalent to $(\FF_n, \sigma)$ . It deforms to a very free curve in $X$ with smaller $-K_X$ degree.
\end{proof}

So we have proved the following.

\begin{lem}\label{lem:balance}
Let $X$ be a smooth projective Fano $4$-fold with pseudo-index at least $2$. Let $C$ be a very free curve with minimal $-K_X$ degree and general in an irreducible component of the moduli space. Assume that $-K_X \cdot C$ is $7$ or $8$. Then the normal bundle is not $\OO(1)\oplus \OO(1)\oplus\OO(a), a=3$ or $4$.
\end{lem}

Now we are ready to finish the proof of Theorem~\ref{thm:lowdegree}.

\begin{proof}[Proof of Theorem~\ref{thm:lowdegree}]
Let $C$ be a very free curve with minimal $-K_X$ degree. If $-K_X \cdot C \leq 7$, then any reducible curve in the curve class $[C]$ and passing through $2$ very general points has at most $3$ irreducible components. If there are $2$ irreducible components, both of them are free. Thus the curve is smoothable to a very free curve. If there are $3$ irreducible components, by Lemma \ref{lem:smoothable},  it is also smoothable. This means that every irreducible component which parametrizes curves passing through two general points has the expected dimension and a general member is an irreducible free curve. Thus the Gromov-Witten invariant $\langle [pt], [pt], [A]^2, \ldots, [A]^2 \rangle^X_{0, \beta}$ is non-zero and enumerative, where $A$ is a very ample divisor.

Now we consider the case $-K_X \cdot C=8$. Since the normal bundle of a general such curve is not $\OO(1)\oplus \OO(1) \oplus \OO(4)$, the deformation of such a curve with two points fixed sweeps out a divisor in $X$. Thus such curves make a positive contribution to the following Gromov-Witten invariant: $\langle [pt], [pt], [A]^3, [A]^2 \rangle^X_{0, C}$. Again $A$ is a very ample divisor as before. We will call the curve and the surface appearing in the constraints $\Gamma$ and $\Sigma$. Our next goal is to prove that such an invariant is non-zero by showing that no reducible curve can contribute to such invariant.

We first choose two very general points. Note that an irreducible rational curve through a general points has $-K_X$ degree at least $2$. By Lemma \ref{lem:smoothable}, we only need to consider the following cases since in all the other cases, the reducible curve lies in an irreducible component whose general point parametrizes a very free curve.
\begin{enumerate}
\item The curve has three irreducible component $C_1, C_2$ and $C_3$ with $-K_X$ degree $2, 4$, and $2$. The degree $2$ curves are free curves.
\item The curve has four irreducible components $C_1$ through $C_4$ with $-K_X$ degree $2$ each. The curves $C_1$ and $C_4$ are free curves. And the deformation of $C_2$ or $C_3$ sweeps out a surface.
\item The curve has four irreducible components $C_1$ through $C_4$ with $-K_X$ degree $2$ each. The curves $C_1$ and $C_4$ are free curves. And the curves $C_2$ and $C_3$ deform in divisors.
\item The curve has four irreducible components $C_1$ through $C_4$ with $-K_X$ degree $2$ each. The curves $C_1$ and $C_4$ are free curves. And either $C_2$ or $C_3$ (or both) deforms to a free curve.
\end{enumerate}

In all cases, the free curve $C_1$ (resp. $C_3$ in case one or $C_4$ in other cases) is fixed and cannot meet a general curve $\Gamma$ once the point is fixed. 

By Lemma \ref{lem:smoothable} and the minimality of the very free curve, $C_1$ and $C_4$ are not connected by a single curve $C_2$ or $C_3$ in the last three cases.  And the curve $C_2$ (resp. $C_3$) is fixed once two points of the curve is fixed by a bend-and-break type argument.  

In the first case, the degree $4$ curve $C_2$ has to meet $\Gamma$. Then the dimension of every irreducible component of the Hom-scheme at $C_2$ is at most $9$. Indeed, if in one irreducible component the deformation of $C_2$ dominates $X$, then that irreducible component has expected dimension $-K_X \cdot C_2+\dim X=8$ since a general deformation is free. Now assume that the deformation is contained in a divisor ( it cannot be contained in a surface since we can choose the two degree $2$ free curves to avoid any finite number of codimension $2$ locus by choosing the two points to be general). Consider the evaluation map
\[
\PP^1 \times \PP^1 \times \PP^1 \times Hom(C_2, X) \to X \times X \times X.
\]
The image in $X\times X \times X$ has dimension $9$. So it suffices to show that the map has fiber dimension $3$. If the fiber dimension is greater than $3$, then we may deform the curve $C_2$ with $3$ points fixed. Then by bend-and-break, it breaks into a reducible or non-reduced curve, still passing through the fixed points. The $3$ points are intersection points of $C_1, C_3$ and $\Gamma$ with the divisor, and thus can be chosen to be general points in the divisor. Then we get a rational curve $D$, necessarily with $-K_X$ degree $2$, passing through $2$ general points in the divisor. The comb consisting of $C_1, C_3$ and $D$ can be smoothed to give a very free curve with smaller $-K_X$ degree by 
 Lemma \ref{lem:smoothable}. This is a contradiction since we assume $C$ is a very free curve with minimal $-K_X$ degree. Thus the Hom-scheme has dimension at most $9$. Then a similar argument using the estimate \ref{eq:1} as in Lemma \ref{lem:smoothable} shows that this reducible curve lies in a component whose general points parametrizes a free curve, in particular, has expected dimension. Thus the first case has no contributions to the Gromov-Witten invariant.

In the remaining cases, note the following.
\begin{lem}
The curve $C_2$ (resp. $C_3$) cannot deform in a surface.
\end{lem}
\begin{proof}
If $C_2$ (resp. $C_3$) deforms in a surface, then it can meet neither $C_1$ nor $C_4$ if we choose the two points to be general. Then $C_3$ (resp. $C_2$) has to connect both $C_1$ and $C_4$. Then by Lemma \ref{lem:smoothable}, there is a very free curve with smaller $-K_X$ degree, which is a contradiction.
\end{proof}
Thus the second case does not contribute to the Gromov-Witten invariant.

\begin{lem}
If $C_2$ (resp. $C_3$) deforms in a divisor, then it cannot pass through two general points in the divisor. In particular, it cannot meet $\Gamma$ and $C_1$ (or $C_4$).
\end{lem}
\begin{proof}
Without loss of generality, assume $C_2$ is connected to $C_1$. If $C_2$ can pass through two general points in the divisor, then it can connect both $C_1$ and a general deformation of $C_1$ (Note $C_1$ has positive intersection with the divisor). Then there is a very free curve with smaller $-K_X$ degree by Lemma \ref{lem:smoothable}. Again a contradiction.
\end{proof}

Thus the third case has no contributions either.

In the fourth case, assume the curve $C_2$ deforms to a free curve. In particular, $C_2$ lies in a component of the Kontsevich moduli space of dimension $3$. Meeting a curve is a codimension $2$ condition. So if we choose the points and the curve $\Gamma$ to be general enough, $C_2$ cannot meet both $C_1$ (or $C_4$) and $\Gamma$. Thus we may assume $C_2$ meet the surface $\Sigma$ and $C_1$. Then $C_2$ has to be a general point in the irreducible component of dimension $3$, thus a free curve. Note that $C_3$ cannot deform in a divisor otherwise $C_3$ cannot meet $\Gamma$. That is, $C_3$ deforms to a free curve and thus the dimension of the Hom-scheme has expected dimension at $C_3$. By Lemma \ref{lem:smoothing}, the union $C_2 \cup C_3 \cup C_4$ deforms to an irreducible free curve. But since $C_1$ is free, the union of $C_1, \ldots, C_4$ deforms to an irreducible free curve. So this reducible curve lies in an irreducible component of expected dimension. Then if we choose the constraints to general, such reducible curves will not contribute to the Gromov-Witten invariant.
\end{proof}

We finish this section by another theorem about very free rational curves of low degree. 

\begin{thm}\label{thm:BalancedCurve}
Let $X$ be a Fano $4$-fold and $f: C \to X$ be an embedded rational curve. Assume that the normal bundle of $C$ is isomorphic to $\OO(2) \oplus \OO(2) \oplus \OO(2)$. Then there is a non-zero Gromov-Witten invariant of the form $\langle [pt], [pt], [pt] \rangle^X_{0, C}$ or $\langle [pt], [pt] \rangle^X_{0, \tilde{C}}$.
\end{thm}

\begin{proof}
We choose the constraints to be three points. And in view of Lemma \ref{lem:smoothing} and \ref{lem:smoothable}, we only need to consider the following possibilities.
\begin{enumerate}
\item There are four irreducible components $C_1$ through $C_4$ with $-K_X$ degree $2, 2, 2$, and $2$. The curves $C_1, C_2$, and $C_3$ are free and pass through one of the chosen general points.
\item There are four irreducible components $C_1$ through $C_4$ with $-K_X$ degree $2, 2, 3$, and $1$. The curves $C_1, C_2$, and $C_3$ are free and pass through one of the chosen general points.
\item There are five irreducible components $C_1$ through $C_5$ with $-K_X$ degree $2, 2, 2, 1, 1$. The curves $C_1, C_2$, and $C_3$ are free and pass through one of the chosen general points.
\end{enumerate}
In the first case, if the Kontsevich moduli space has the expected dimension at the curve $C_4$, then the union of $C_1, \ldots, C_4$ can be smoothed to a free curve by Lemma \ref{lem:smoothable}. Thus it lies in an irreducible component of expected dimension. So this will not happen if we choose the $3$ points to be general. So the dimension of the Kontsevich moduli space at the point $[C_4]$ has strictly higher dimension than the expected dimension. This means that $C_4$ only deforms in a divisor (it cannot deform in a surface otherwise we can choose the curves $C_1, \ldots, C_3$ to miss it). Then the bend-and-break argument shows that there is a rational curve $D$ of $-K_X$ degree $1$ and passing through $2$ general points in the divisor. It can connect two of the three free curves, say, $C_1$ and $C_2$. Now take the reducible curve $C_1\cup D \cup C_2$. Then the Gromov-Witten invariant $\langle [pt], [pt]\rangle^X_{0, [C_1+C_2+D]}$ is non-zero. In fact, an irreducible curve in this curve class meeting two general points are very free and the moduli space is smooth of expected dimension at this point. And the only possible reducible curve in this class that can pass through $2$ general points is either of the form $C\cup C'$, where $C$ and $C'$ are free curves, or $C+C'+D$, where $C$ and $C'$ are free curves and $-K_X \cdot D=1$. The first case is impossible if we choose the points to be general since every irreducible component of the moduli space has expected dimension at this point. In the second case $D$ only deforms in a divisor ($D$ cannot deform in a surface otherwise we could choose the constraints so that $C$ and $C'$ avoid the surface). Again by bend-and-break, the curve $D$ cannot deform once $C$ and $C'$ are fixed. Thus the moduli space has expected dimension at this point, although possibly non-reduced. So the Gromov-Witten invariant is always positive.

The second case cannot happen if we choose the points to be general. In fact, the degree $1$ curve deforms in a divisor, and it has to connect all the three free curves. By the bend-and-break argument, it cannot deform once the two degree $2$ curves are fixed. But then the degree $3$ curve will miss this curve once the point is chosen to be general.

In the third case one of the degree $1$ curve will connect two degree $2$ curves. Then the union of them gives a non-zero Gromov-Witten invariant of the form $\langle [pt], [pt]\rangle^X_{0, \tilde{C}}$ by the same argument as the first case.
\end{proof}

\section{proof of Theorem~\ref{thm:main}}
Throughout the section, let $X$ be a smooth projective Fano $4$-fold of pseudo-index at least $2$.

The following lemma is taken from \cite{AG_Fano}, Lemma 7.2.16 (ii). The statement is due to Wi{\'s}niewski \cite{Wis} but stated only for Fano fourfold of index $2$. However the same proof works in our setting as well. We reproduce the proof here for the convenience of the reader.

\begin{lem}\label{lem:3}
Let $X \to Y$ be a contraction of a $K_X$-negative extremal face. If there is a $3$-dimensional fiber, then $X$ is a $\PP^1$ bundle over a smooth projective Fano $3$-fold.
\end{lem}
\begin{proof}
Let $D$ be a $3$-dimensional fiber. There is an extremal ray $R$, which does not lie in the extremal face corresponding to the contraction $X \to Y$ and spanned by the class of a rational curve $C$ such that $C \cdot D >0$. Then the corresponding contraction of $R$ have fiber dimension at most $1$. Indeed if there is a $2$-dimensional fiber $E$, then $D \cap E$ is non-empty and has a $1$-dimensional component. But the class of this component has to lie in both $R$ and the extremal face, which is impossible. Then the contraction corresponding to $R$ is either a blow-up along a smooth surface or a conic bundle over a smooth $3$-fold by \cite{Ando}. The former case is impossible because of the assumption on the pseudo-index. For the same reason, the conic bundle is a $\PP^1$-bundle. The base $3$-fold is Fano by a result of Wi{\'s}niewski \cite{WisF} (c.f. Corollary 7.2.13 in \cite{AG_Fano}).
\end{proof}

\begin{lem}\label{lem:2}
Let $X \to Y$ be a contraction of a $K_X$-negative extremal ray. If the map contracts a divisor to a curve, then $Y$ is smooth and $X$ is the blow-up along a smooth curve. Furthermore, $Y$ is a Fano $4$-fold of pseudo-index at least $2$.
\end{lem}

\begin{proof}
Here we need the description of extremal contractions as in \cite{AWContraction2}, see also \cite{AMContraction} for the summary of fourfold contractions. By Theorem 4.1.3 in \cite{AMContraction}, we have the following possibilities.
\begin{enumerate}
\item The exceptional divisor $E$ is a $\PP^2$-bundle and the normal bundle of each fiber in $X$ is
either $\OO(-1)\oplus \OO$ or $\OO(-2) \oplus \OO$. When the normal bundle is $\OO(-1)\oplus \OO$, the contraction $X\to Y$ is the inverse of the blow-up of a smooth curve in a smooth variety $Y$.
\item The exceptional divisor $E$ is a quadric bundle and the general fiber is irreducible and
isomorphic to a two dimensional, possibly singular, quadric. The normal bundle of each fiber is
$\OO(-1)\oplus \OO$.
\end{enumerate}
Since the pseudo-index of $X$ is at least $2$, only the first case with normal bundle $\OO(-1) \oplus \OO$ can occur. So $Y$ is smooth and the map from $X$ to $Y$ is just the inverse of the blow-up along a smooth curve $C$.  Let $D$ be an irreducible curve in $Y$ not supported on $C$ and $\tilde{D}$ its strict transform in $X$. Then 
\[
-K_Y\cdot D = -K_X \cdot \tilde{D}+E \cdot \tilde{D}\geq -K_X \cdot \tilde{D}>0.
\] 
Since $-K_X$ is ample, $(-K_X)^3 \cdot E>0$. Thus $-K_Y \cdot C > 4g(C)-4$. So $Y$ is Fano with pseudo-index at least $2$ if $C$ is not a rational curve. If $C$ is rational, assume the normal bundle of $C$ in $Y$ $N_{C/Y}$ is isomorphic to $\OO(a)\oplus \OO(b) \oplus \OO(c)$ with $a\geq b\geq c$. Then $E \cong \PP_C(\OO(a)\oplus \OO(b) \oplus \OO(c))$. Let $\tilde{C}\subset E $ be the section corresponding to the inclusion $\OO(a) \to \OO(a)\oplus \OO(b) \oplus \OO(c)$. Then
\[
2 \leq -K_X \cdot \tilde{C}=2+a+b+c-2a.
\]
Thus $2+c \geq 2+b+c-a \geq 2$, and $a\geq b \geq c \geq 0$. Therefore $Y$ is also a Fano variety of pseudo-index at least $2$.
\end{proof}

\begin{lem}\label{lem:birational}
Let $X$ be a Fano $4$-fold of pseudo-index at least $2$ and $f:Y \to X$ a birational morphism. Assume that $X$ has a non-zero Gromov-Witten invariant with two point-insertions which counts irreducible very free rational curves. Then $Y$ also has a non-zero Gromov-Witten invariant with two point-insertions which counts irreducible very free rational curves.

\end{lem}

\begin{proof}
The images under $f$ of the exceptional divisors have codimension at least 2 in $X$. Thus the very free curve in $X$ can be deformed away from them. We can also choose the constraints to be away from the centers. Then the very free curves in $X$ meeting all the constraints are all away from the images of exceptional divisors. Observe that the image of any curve satisfying the constraints in $Y$ also satisfies the constraints in $X$. Thus the images are irreducible curves not intersecting the exceptional locus. Then it follows that no components are contracted by the map $f: Y \rightarrow X$ and the curves in $Y$ that can meet these constraints are again irreducible very free curves.
\end{proof}

Finally we need the following result, which is a special case of \cite{SymGHS} Theorem 1.14.

\begin{thm}\label{thm:conic}
Let $Z$ be a Fano threefold and $X\rightarrow Z$ be a morphism whose general fiber is $\PP^1$. Then $X$ is has a non-zero Gromov-Witten invariant with two point insertions.
\end{thm}

\begin{proof}[Proof of Theorem \ref{thm:main}]
First consider the case when $X$ has Picard number $1$. Let $C$ be a minimal free curve. If $-K_X \cdot C=2$, then there is a very free curve of $-K_X$ degree $8$ by the quotient construction in \cite{Kollar96}, Chap. IV, Theorem 4.13. If $-K_X \cdot C=4$, then there is a very free curve of $-K_X$ degree $8$ by deforming two such curves. If $-K_X \cdot C=5$, the $N_{C/X} \cong \OO(1)\oplus \OO(1) \oplus \OO(1)$ and $C$ is very free. In any case, we have a very free curve of $-K_X$ degree at most $8$ and the statement follows from Theorem \ref{thm:lowdegree}.

If $-K_X \cdot C=3$, then the quotient construction (loc. cit.) gives a very free curve of $-K_X$ degree $9$. Note that in this case any reducible curve in this curve class that passes through two general points can be smoothed to a very free curve by Lemma \ref{lem:smoothable}. So every irreducible component of the moduli space of rational curves in this curve class containing $2$ general points has the expected dimension. And the contribution to the Gromov-Witten invariant $\GWpp$ is positive. Thus the invariant is non-zero.

Now consider the case when $X$ has higher Picard number. By Lemma \ref{lem:3} and Theorem \ref{thm:conic}, we may assume that there is no contraction of fiber dimension $3$. Note that there is no birational contraction of a divisor to a surface by the assumption on the pseudo-index. Later, all the Gromov-Witten invariants we proved to be non-zero will be the actual count of irreducible very free curves passing through $2$ general points and other constraints. Thus by Lemma \ref{lem:2} and \ref{lem:birational}, we can also assume that there is no birational contraction of a divisor to a curve. There are no flips since the pseudo-index is at least $2$.

So from now on we will assume all the contractions are of fiber type and have fiber dimension at most $2$. There are two cases. 
\begin{enumerate}
\item All fiber type contractions are del Pezzo fibrations.
\item There exist contractions whose general fibers are isomorphic to $\PP^1$.
\end{enumerate}

In the first case, let $F_1$ and $F_2$ be general fibers of two different contractions. The surfaces $F_1$ and $F_2$ are del Pezzo surfaces. Thus there is a rational curve $C_i$ in $F_i$, such that $N_{C_i/X}\cong \OO(a_i) \oplus \OO \oplus \OO, 2 \geq a_i \geq 1$ for $i=1, 2$. The two surfaces $F_1$ and $F_2$ have to intersect at finitely many points. So we can glue deformations of the curves $C_1$ and $C_2$ at one of the intersection points. We may choose the two surfaces to be general so that these deformations of $C_1$ and $C_2$ are still irreducible free curves and passing through a general point in $X$. Then a general deformation of this reducible curve is a very free curve of $-K_X$ degree at most $8$.

In the second case, choose an extremal ray $R$ whose corresponding contraction has relative dimension $1$. Let $C$ be a free curve whose curve class lies in $R$. There is a contraction $X \to B$ that contracts all the other extremal rays except the extremal ray $R$. Note that the base $B$ has Picard number $1$.

If the base has dimension $1$, then a general fiber is a Fano $3$-fold. It is show in \cite{SymRC} that every Fano $3$-fold has a very free curve with anti-canonical degree at most $6$. Then it is easy to see that there is a very free curve of $X$ with $-K_X$ degree at most $8$ by smoothing the union of a free curve whose class lies in $R$ and a very free curve in the fiber. 

Now consider the case $\dim B=2$. If there is a $3$-dimensional fiber $F$ of the contraction, then the result follows from Lemma \ref{lem:3} and Theorem \ref{thm:conic}. So from now on assume that all fibers have dimension $2$. Then the image of a general deformation of $C$ lies in the smooth locus of $B$. Thus the image is a Cartier divisor, necessarily ample. Take two general such curves $C_1$ and $C_2$, each containing a general point. Their images in $B$ intersect at points whose fibers are smooth del Pezzo surfaces. Take a curve $D$ which is a very free curve in the fiber with $-K_X$ degree at most $4$. 
Then there is a very free curve of $-K_X$ degree at most $8$ by smoothing the reducible curve $C_1 \cup D \cup C_2$. 

If $\dim B=3$, then the contraction is given by another extremal ray $R'$ and $X$ has Picard number $2$. We have two families of free rational curves $U \leftarrow \mathcal{C} \to X$ and $U' \leftarrow \mathcal{C'} \to X$ such that a general member of $\mathcal{C}$ resp. $\mathcal{C}'$ is a rational curve whose curve class lies in $R$, resp. $R'$. Now consider the flat quotient of $X$ by the relation generated by these two families as in Theorem 4.13, Chapter IV of \cite{Kollar96}. By this theorem, there is a dominant rational map $X \dashrightarrow Z$ such that two general points in the fiber are connected by a chain of curves in the two families with length at most $4$. Since $X$ has Picard number $2$ and the rational map contacts both curves in $R$ and $R'$, $Z$ is forced to be a point. Therefore, by smoothing the reducible curve connecting two general points, we get a very free curve with $-K_X$ degree at most $8$. 

In any case, there is a very free curve of $-K_X$ degree at most $8$ and the theorem follows from Theorem \ref{thm:lowdegree}.
\end{proof}

\section{Extremal rays and Gromov-Witten invariants}
In this section, let $X$ be a smooth projective variety of dimension $4$ and $X \to Y$ be the contraction of a $K_X$-negative extremal ray.

\subsection{Fano fibrations}

In this case, a general fiber $F$ is a smooth Fano variety. Then choose a minimal free rational curve $C$ in a general fiber. Then there is a non-zero Gromov-Witten invariant of the form $\langle [pt], \ldots \rangle^X_{0, [C]}$ (Theorem 4.2.10 \cite{KollarUni}, Proposition G, Proposition 4.9 \cite{RuanUni}). 

One can even do better. By the result of \cite{SymRC}, every Fano variety of dimension at most $3$ has a non-zero Gromov-Witten invariant of the form $\GWpp$. Then let $A$ be a very ample divisor and choose the curve class to the image of $\beta$ under the inclusion of $F$ into $X$. Then it is easy to see that there is a non-zero Gromov-Witten invariant of the form $\langle [pt], [A]^{d-1}, \ldots \rangle_{0, \beta}^X$, where $d$ is dimension of a general fiber.

\subsection{Flip}
By Theorem 1.1 of \cite{KawamataSmallContraction}, the exceptional locus $E$ is a disjoint union of $\PP^2$s and the normal bundle of each irreducible component
is $\OO(-1)\oplus \OO(-1)$. Let $L$ be a line in $E\cong \PP^2$. Then $\langle [E], [E] \rangle^X_{0, L}$ is the number of irreducible components.

\subsection{Divisor to a surface} 
In this case a general fiber $F$ of the exceptional divisor $E$ is a $\PP^1$ whose normal bundle is $\OO\oplus \OO \oplus \OO(-1)$. Then $\langle [F] \rangle^{X}_{[F], 0}=F \cdot E =-1$.
\subsection{Divisor to curve}
We know that the curve in $Y$ is smooth and $X$ is the blow-up of $Y$ along the curve. There are $3$ possibilities (c.f. Thereom 4.1.3 part 3, \cite{AMContraction}).
\begin{enumerate}
\item The exceptional divisor $E$ is a $\PP^2$-bundle and the normal bundle of each fiber in X is
$\OO(-1)\oplus \OO$.
\item The exceptional divisor $E$ is a quadric bundle and the general fiber is irreducible and
isomorphic to a two dimensional, possibly singular, quadric. The normal bundle of each fiber is
$\OO(-1)\oplus \OO$.
\item The exceptional divisor $E$ is a $\PP^2$-bundle and the normal bundle of each fiber in X is
$\OO(-2)\oplus\OO$.
\end{enumerate}
Let $L$ be a line in a general fiber of $E$. Then it is easy to check that in the first case, $\langle [L], [L] \rangle^X_{0, [L]}=1$. In the second case $\langle [L] \rangle^X_{0, [L]}=-2$. The third case need an excess intersection calculation, which is essentially done in \cite{RuanExtremal} since locally the variety $Y$ is the product of a smooth curve with a singular $3$-fold considered by Ruan. By the proof of Theorem 5.1 (loc. cit.), in particular the computations on p. 422, we know 
$\langle [L] \rangle^X_{0, [L]}=2$.
\subsection{Divisor to point}
This is the most complicated case. And the results are not needed elsewhere in the paper. The classification is in \cite{FujitaSingularDelPezzo}, \cite{FuitaClassification}, \cite{BContraction}, and \cite{BContraction2}. We summarize as follows.

\begin{enumerate}
\item The exceptional divisor $E$ is $P^3$, with normal bundle
$\OO(-a)$ and $1\leq a \leq 3$.
\item  The exceptional divisor $E$ is a (possibly singular) three dimensional quadric,
with normal bundle $\OO(-a)$ and $1 \leq a \leq2$.
\item The pair $(E; E\vert_E)$ is a  (possibly singular) del Pezzo threefold, i.e. $\OO_E(-E)$ is ample and $-K_E\cong \OO_E(-2E)$. For the classification, see \cite{FujitaSingularDelPezzo}.
\item The exceptional divisor is non-normal.
\end{enumerate}
\begin{rem}
The first two cases are in Proposition 2.4 of \cite{BContraction}. Also it is conjectured the non-normal case does not occur (Remark 4.1.4, \cite{AMContraction}). It seems to the author that the list in Supplement 2.6 of \cite{BContraction} has overlooked the cases where the del Pezzo threefolds are cones over a singular del Pezzo surface or an elliptic curve. 
\end{rem}
\begin{prop}\label{case1}
Let $L$ (resp. $H$) be a line (resp. hyperplane) in $E\cong \PP^3$. Assume $\OO_E(E) \cong \OO(-a)$. Then
\begin{enumerate}
\item If $a=1$, $\langle [L], [L] \rangle^X_{0, [L]}=1$. 
\item If $a=2$, $\langle [L], [H] \rangle^X_{0, [L]}=-4$.
\item If $a=3$, $\langle [L] \rangle^X_{0, [L]}=-6$.
\end{enumerate}
\end{prop}
\begin{proof}
The first case is obvious.
 
In the second and third case, the moduli space $\overline{M}_{0, 0}(X, [L])$ is isomorphic to $G(2, 4)$ and the universal family is $\mathcal{C} \cong \PP(S) \to G(2, 4)$, where $\pi: S \to G(2, 4)$ is the universal rank $2$ subbundle. And the natrual map $ev: \mathcal{C} \to \PP^3\cong E \subset X$ is the evaluation map. We have the following exact sequences:
\[
0 \to T_E \to T_X \to \OO_E(-a) \to 0,
\]
\[
0 \to T_\pi \to ev^* T_X \to N \to 0,
\]
where $T_\pi$ is the relative tangent sheaf.
And the obstruction bundle $Obs$ is $ R^1\pi_*(ev^*N))$.
We need to know the Euler class of the obstruction bundle. Note that $TG(2,4) \cong R^0\pi_*(N)$. Thus the result follows from a routine calculation using Grothendieck-Riemann-Roch. Denote by $h$ the hyperplane class of $\PP^3$ and $c$ the first Chern class of $\OO_{\PP(S)}(1)$. Then $ev^*(h)=c$. We will also use $\sigma_{*}$ for both the Schubert classes on $G(2,4)$ (c.f. Section 1.5 in \cite{GriffithsHarris}) and their pull-backs to $\PP(S)$ since no confusions are likely. It is straightforward to verify the followings.

\begin{align*}
&Ch(ev^* T_X)\\
=&ev^*(Ch (T_E) +Ch(\OO(-a)))\\
=&ev^*(4(1+h+\frac{h^2}{2}+\frac{h^3}{6}+ \ldots)-1+1-ah+\frac{a^2h^2}{2}-\frac{a^3h^3}{6}+\ldots))\\
=&4+(4-a)c+\frac{4+a^2}{2}c^2+\frac{4-a^3}{6}c^3+\ldots\\
\end{align*}
\[
Ch(T_\pi)=1+(2c-\sigma_1)+\frac{(2c-\sigma_1)^2}{2}+\ldots,
\]
\[
Td(T_\pi)=1+(2c-\sigma_1)+\frac{(2c-\sigma_1)^2}{12}+\ldots,
\]
\begin{align*}
&Ch(N)\\
=&Ch(ev^*T_X)-Ch(T_\pi)\\
=&3+(2-a)c+\sigma_1+\frac{a^2+4}{2}c^2-\frac{(2c-\sigma_1)^2}{2}+\frac{4-a^3}{6}c^3-\frac{(2c-\sigma_1)^3}{6}+\ldots
\end{align*}
Thus
\[
Ch(\pi_*[N])=(5-a)+\frac{a^2+4}{2}\sigma_1+\frac{12-a+3a^2-2a^3}{12}\sigma_2+\frac{4+a-a^2}{4}\sigma_{1,1}+\ldots.
\]
Since
\[
Ch(TG(2,4))=Ch(S^{*}\otimes (V/S))=4+4\sigma_1+\sigma_{1,1}+\sigma_2+\ldots,
\]
we have
\[
Ch_1(Obs)=\frac{a-a^2}{2}\sigma_1,
\]
\[
Ch_2(Obs)=\frac{a-3a^2+2a^3}{12}\sigma_2+\frac{a^2-a}{4}\sigma_{1,1}.
\]
Thus when $a$ is $2$, the Euler class of the obstruction bundle is $-\sigma_1$. And when $a$ is $3$, the Euler class of the obstruction bundle is $2\sigma_2+3\sigma_{1,1}$. Then the Gromov-Witten invariants are just the intersection numbers $(-2\sigma_2) \cdot (-2\sigma_1) \cdot (-\sigma_1)=-4$ and $(-3\sigma_2)\cdot (2 \sigma_2+3 \sigma_{1, 1})=-6$.
\end{proof}
\begin{prop}
Let $E$ be a smooth quadric with normal bundle $\OO_E(E) \cong \OO(-a)$. And let $L$ (resp. $H$) be a line (resp. a hyperplane section) in $E$. Then when $a$ is $1$, $\langle [L], [H] \rangle^X_{0, [L]}=1$. When $a$ is $2$, $\langle [L] \rangle^X_{0, [L]}=4$.
\end{prop}
\begin{proof}
When $a$ is $1$, the Gromov-Witten invariant is enumerative and is just the intersection number $4\sigma_{1,1}\cdot \sigma_1 \cdot (-\frac{1}{2}\sigma_1) \cdot (-\frac{1}{2}\sigma_2)=1$. When $a$ is $2$, we use an excess intersection computation similar as above. The moduli space $\overline{M}=\overline{M}_{0, 0}(X, [L])$ is isomorphic to the zero locus of a section of $Sym^2 S^*$ over $G(2, 5)$ and the universal family $\mathcal{C} \to \overline{M}$ is the restriction of $\PP(S)\to G(2,5)$, where $\pi: S \to G(2, 5)$ is the universal rank $2$ subbundle. Let $ev: \mathcal{C} \to Q \cong E \subset X$ be the evaluation map, where $Q \subset \PP^4$ is a smooth quadric hypersurface.
\[
0 \to T_Q(\cong T_E) \to T_{\PP^4} \to N_{Q/\PP^4} \to 0,
\]
\[
0 \to T_E \to T_X \to \OO_E(-2) \to 0
\]
\[
0 \to T_\pi \to ev^* T_X \to N \to 0.
\]
And the obstruction bundle $Obs$ is $ R^1\pi_*(ev^*N))$. As before, let $c$ denote the first Chern class of the relative $\OO_{\PP(S)}(1)$, which is the pull-back of $\OO_E(1)$ via the evaluation map. We list the results below.
\begin{align*}
&Ch(ev^* T_X)\\
=&ev^*(Ch (T_E) +Ch(\OO(-2)))\\
=&ev^*(Ch (T_{\PP^5})-Ch(\OO_{E}(2)+Ch(\OO(-2)))\\
=&4+c+\frac{5}{2}c^2+\ldots\\
\end{align*}
\[
Ch(T_\pi)=1+(2c-\sigma_1)+\frac{(2c-\sigma_1)^2}{2}+\ldots,
\]
\[
Td(T_\pi)=1+(2c-\sigma_1)+\frac{(2c-\sigma)^2}{12}+\ldots,
\]
\[
Ch(N)=Ch(ev^*T_X)-Ch(T_\pi)=3+\sigma_1-c+\frac{5}{2}c^2-\frac{(2c-\sigma_1)^2}{2}+\ldots
\]
Thus
\[
Ch(\pi_*[N])=2+3\sigma_1+\ldots.
\]
\[
Ch(TG(2,5))=Ch(S^{*}\otimes (V/S))=6+5\sigma_1+\ldots,
\]
\[
Ch(T_{\overline{M}})=Ch(TG(2,5))-Ch(Sym^2S^*)=3+2\sigma_1+\ldots.
\]
\[
Ch(Obs)=1-\sigma_1+\ldots
\]
Thus the Euler class of the obstruction bundle is $-\sigma_1$. Then one can easily check the Gromov-Witten invariant is just $4\sigma_{1,1}\cdot \sigma_1 \cdot (-\sigma_1) \cdot (-\sigma_2)=4$.
\end{proof}
We now turn to the del Pezzo threefold case.
\begin{prop}\label{prop:rational}
Let $(E, \OO(1))$ be a del Pezzo threefold with rational singularities. Then there are only finitely many lines $L$ (i.e. $L \cdot \OO(1)=1$) that pass through a general point.
\end{prop}

\begin{proof}
By definition, a del Pezzo threefold has locally complete intersection singularities. Thus having rational singularities is the same as having canonical singularities (Corollary 5.24, \cite{KM98}). There is a $3$-fold $E'$ with terminal singularities and a birational morphism $f: E' \to E$ such that $f^*K_E=K_{E'}$. We may also choose $f$ to be an isomorphism over the smooth locus of $E$. To see this, choose a resolution of singularities of $E$ which is an isomorphism over the smooth locus of $E$ and then run the relative MMP over $E$. Note that families of lines through a general point in $E$ are in one to one correspondence with families of generically irreducible curves $C_t$ with $-K_{E'}\cdot C_t=2$ through a general point in $E'$. Thus it suffices to show that there are only finitely many such curves through a general point in $E'$. But every irreducible curve $C$ through a general point in $E'$ with $-K_{E'}\cdot C=2$ lies in the smooth locus of $E'$ by Lemma 5.2 in \cite{SymRC}. Hence such curves have no first order deformation (as a stable map) with one point fixed. So there are only finitely many such curves.
\end{proof}

\begin{prop}\label{prop:isolated}
If $E$ has at worst isolated singularities, then there are only finitely many lines through a general point.
\end{prop}
\begin{proof}
Let $p$ be a general point. Assume that there is a $1$-dimensional family of lines $L_t$ containing $p$. Then all the $L_t$ has to contain a singular point $q$. Otherwise since $-K_E \cdot L_t=2$, a line in the smooth locus has no deformation with one point fixed. But then this one parameter family has two points ($p$ and $q$) fixed. So by bend-and-break, this family has to degenerate. But there is no degeneration of lines. Thus there are only finitely many lines through $p$.
\end{proof}

By Propositions \ref{prop:rational}, \ref{prop:isolated}, for a line $L$ in $E$, there is a non-zero Gromov-Witten invariant $\langle [C] \rangle^X_{0, [L]}$ for some curve $C$ intersecting $E$ at general points since in this case the Gromov-Witten invariant is just counting the number of lines meeting $C$.

\section{Uniruled $4$-folds whose second Betti number is $2$}
Let $X$ and $Y$ be two smooth projective manifolds that are symplectic deformation equivalent. Assume $X$ is rationally connected. The first lemma deals with the maximal rationally connected (MRC) quotient of $Y$.

\begin{lem}\label{lem:mrc}
Let $Y \dashrightarrow B$ be the MRC quotient of $Y$. Then $\dim Y -\dim B \geq 2$.
\end{lem}
\begin{proof}
This essentially follows from the proof of Theorem 1.6 in \cite{SymRC}. In fact, by the result of Koll{\'a}r and Ruan (Theorem 4.2.10, \cite{KollarUni}, Proposition G, Proposition 4.9, \cite{RuanUni}), $Y$ is uniruled. Thus the MRC quotient has at least $1$ dimensional fiber. If the fiber is $1$ dimensional, then the argument in the proof of Theorem 1.6 in \cite{SymRC} shows that there exist more rational curves in $Y$ that need to be quotient out to form the MRC quotient, thus a contradiction.
\end{proof}

A natural question to ask is whether Hodge numbers are symplectic deformation invariants. This question is closely related to Koll{\'a}r's conjecture in the following way.

Koll{\'a}r's conjecture predicts that rational connectedness is a symplectic invariant in this setting. There is an obvious obstruction for a variety to be rationally connected. Namely, for a rationally connected variety, the Hodge numbers $h^{0,p}=h^{p,0}$ are zero. We first show that this obstruction vanishes in dimension $4$.

From now on, assume $X$ and $Y$ are $4$-folds. 

\begin{lem}\label{lem:hodgenumber}
The Hodge numbers of $Y$ are the same as those of $X$. In particular, $h^{0,p}(Y)=h^{p,0}(Y)=0$.
\end{lem}
\begin{proof}
We first show that $H^0(Y, \Omega_Y^{p})=0$ for $p=3, 4$. By Lemma \ref{lem:mrc}, the MRC quotient map $Y \dashrightarrow B$ has fiber dimension at least $2$. Moreover, the general fiber is a proper rationally connected subvariety of $Y$. Thus there is a rational curve $f: \PP^1 \to Y$ such that
\[
f^* T_Y \cong \oplus_{i=1}^4 \OO(a_i)
\]
and
\[
a_1 \geq \ldots, \geq a_4 \geq 0, a_1 \geq a_2 \geq 1.
\]
So $\Omega_Y^p$ restricts to negative vector bundles on such curves if $p=3$ or $4$, and in particular has no sections. Since the deformation of such curves cover a dense open subset of $Y$, $\Omega_Y^p$ has no global sections for $p=3, 4$.

Under symplectic deformations, the cohomology classes of the Chern classes (of the tangent bundle) remain the same. Then by Hirzebruch-Riemann-Roch, so is the holomorphic Euler characteristic $\chi(X, \Omega^p)=\sum (-1)^i h^{p, i}$. Notice that $h^{0,0}=1$, $h^{1,0}=h^{0,1}=\frac{1}{2}b_1$ and $\sum_{p+q=i} h^{p, q}=b_i$ are even topological invariants. An easy calculation using these equalities shows that the Hodge numbers of $Y$ are the same as those of $X$.
\end{proof}

\begin{rem}
Using a similar argument on the holomorphic Euler characteristic and the topolgical pairing on $H^3$, one can show that the Hodge numbers are symplectic deformation invariants in dimension $3$.
\end{rem}

Now assume $b_2(X)=b_2(Y)=2$. Note that in this case we have $h^{0, 2}=h^{2, 0}=0$ for any K{\"a}hler manifold by Hodge symmetry. We want to show that $Y$ is also rationally connected.
\begin{lem}\label{lem:structure}
Let $Z$ be a smooth projective variety of dimension $4$ which is uniruled but not rationally connected. If the second Betti number of $Z$ is $2$, then the MMP for $Z$ consists of a finite number of flips $Z=Z_0 \dashrightarrow Z_1 \dashrightarrow \ldots \dashrightarrow Z_n$ and a fiber type contraction $Z_n \to B$, which realizes the MRC quotient of $Z$. Furthermore, either $\dim B=3$ or $Z=Z_n$ (i.e. no flips are needed).
\end{lem}

\begin{proof}
Since $Z$ is uniruled, $K_Z$ is not nef and there is a contraction $Z \to W$.

If the contraction is divisorial, then $W$ is a $\QQ$-Fano $4$-fold, thus rationally connected, a contradiction to our assumption.

If the contraction is a fiber type contraction, then $W=B$ is non-uniruled. Otherwise there are rational curves passing through a general point of $Z$ that are not contracted. We can make such curves to be free rational curves. The fibers over the singular points of $B$ has codimension at least $2$ since the Picard number of $Z$ is $2$. Thus we can deform such a curve to be away from such fibers. Then the image of these curves will be contained in the smooth locus of $B$. Since $B$ is $\QQ$-factorial and has Picard number $1$, then an argument similar to the proof of Corollary 4.14, Chap. IV \cite{Kollar96} shows that $B$ is rationally connected. So $Z$ is also rationally connected, a contradiction.

If the contraction is a flipping contraction, by a result of Kawamata (Theorem 1.1 \cite{KawamataSmallContraction}), there is a flip $Z \dashrightarrow Z_1$ such that $Z_1$ is also smooth (and uniruled). The flips will terminate since $b_4$ drops by one after each flip. Then one can apply the above arguments to the last variety $Z_n$. So there is a fiber type contraction $Z_n \to B$.

By the description of $4$-fold flips, there is a rational curve in $Z_n$ which is not contracted to a point in $B$. Thus if $\dim B=1$ and $Z$ is not rationally connected, then there are no flips in the running of the MMP. If $\dim B=2$, then the fibration $Z_n \to B$ is equi-dimensional. And thus by Corollary 1.4 in \cite{ViewOnContractions}, the base $B$ is smooth. Note that the Picard number of $B$ is $1$. Thus the first Betti number $b_1(B)$ is $0$ or the Albanese map $B \to Alb(B)$ is finite. Also $H^0(B, K_B)=0$ otherwise a non-zero section of $K_B$ pulls back to a non-zero section of $\Omega^2_Z$. But $b_2(Z)=2$ by assumption and $h^{2,0}=h^{0,2}=0$. Therefore $b_2(B)=h^{1,1}(B)=1$. If $b_1(B)=0$, then $B$ has the same rational homology as $\PP^2$. Thus $B$ is either $\PP^2$ or a ``fake projective space". In the first case $Z$ is rationally connected. For the latter case, $B$ is a quotient of a $2$ dimensional ball by a discrete group action by a theorem of Yau (Theorem 4, \cite{YauCalabi}). In particular, $B$ contains no rational curves. But this is impossible since the flipped rational curve is not contracted to a point in $B$. If the Albanese map is finite, then $Alb(B)$ is at least $2$ dimensional. Thus $H^0(B, \Omega_B) \geq 2$. Then taking wedge product gives a nontrivial element of $H^0(B, K_B)$. This is impossible as shown above. Thus there is no flip in the running of the MMP.
\end{proof}

Now we return to the study of the MRC quotient of $Y$.
\begin{lem}\label{lem:finestructure}
If $Y$ is not rationally connected, then there is a contraction of an extremal ray $R_Y$, $Y \to B$ with $\dim B=2$.
\end{lem}
\begin{proof}
By Lemma \ref{lem:mrc}, if the end product of the MMP for $Y$ is a contraction to a $3$-fold $B$, then $B$ is uniruled and thus $Y$ is rationally connected. Thus by Lemma \ref{lem:structure}, there is a contraction $Y \to B$ with $\dim B \leq 2$. If $B$ is a curve, then $B$ has to be $\PP^1$ otherwise $b_1(X)=b_1(Y)\geq 2 h^0(B, \Omega_B) \geq 2$. But then $Y$ is rationally connected, a contradiction.
\end{proof}

\begin{lem}\label{lem:extremalray}
Let $V$ and $V'$ be two smooth projective $4$-folds which are symplectic deformation equivalent. Assume $b_2(V)=b_2(V')=2$. Let $V \to S$ be an extremal contraction corresponding to an extremal ray $R$. Assume $\dim S \leq 2$. Then there is a contraction of an extremal ray $R'$ of $V'$ such that
the extremal ray $R$ in $H_2(V, \RR)$ is the same as the extremal ray $R'$ in $H_2(V', \RR)$ under the identification $H_2(V, \RR)\cong H_2(V', \RR)$. Furthermore, the two contractions are fiber type contractions with the same fiber dimension.
\end{lem}

Note that in this case $Pic(V) \otimes \QQ \cong H^2(V, \QQ) \cap H^{1,1}(V, \CC)$ and $Pic(V') \otimes \QQ \cong H^2(V', \QQ) \cap H^{1,1}(V', \CC)$. So the cone of effective curves naturally lies in $H_2(V, \QQ)$ and $H_2(V', \QQ)$.

\begin{proof}
By Corollary \ref{cor:extremal}, the ray $R$ stays effective under deformations of tamed almost complex structures. In particular, $R$ lies in the effective cone of curves of $V'$.

We first claim that $V'$ has no extremal contractions which contract a divisor to a point. Otherwise write the exceptional divisor as $E$ and the pull-back of an ample divisor on the target as $H$. Clearly $H \cdot E=0$. Given a line bundle $L$, consider the Lefschetz map:
\[
c_1(L)^2: H^2(V', \QQ) \to H^6(V', \QQ).
\]
The classes $C=H^3$ and $D=E^3$ span $H^6(V, \QQ)$. Write $L=xH+yE$ and in terms of this basis, the linear map can be written as:
\[
\left(
  \begin{array}{cc}
    x^2 & 0 \\
    0& y^2 \\
  \end{array}
\right)
\]
Thus the map is not an isomorphism if and only if $L$ is a multiple of $H$ or $E$. Note that this map is purely topological. Thus the same is true for $V$. But $V$ has a contraction to $B$. Let $A$ be the pull-back of an ample divisor on $S$. Then $c_1(A)^2$ is not an isomorphism since $A^3=0$. Thus the cohomology class $c_1(A)$ is a multiple of $c_1(H)$ or $c_1(E)$. But $c_1(A)^3=0$ and $c_1(H)^3 >0, c_1(E)^4<0$.

Therefore all the $K_{V'}$-negative extremal rays lie in the effective cone of $V$ under the natural identification of $H_2(V, \RR)\cong H_2(V', \RR)$ by Corollary \ref{cor:extremal}. Then $V$ and $V'$ has to share a common extremal ray.

To prove that the two contractions are both fiber type contractions, notice that if one of them is a fiber type contraction, then there is a non-zero Gromov-Witten invariant of the form $\langle [pt], \ldots \rangle_{0, \beta}$ for some $\beta$ in the extremal ray. Thus the contraction has to be of fiber type on the other one. Note that the pull back of an ample divisor from the base spanns the unique line in $H^2(V, \QQ)$ (and $H^2(V', \QQ)$) which has zero intersection number with $R$. Thus one can read off the fiber dimension by looking at the self-intersection numbers of a non-zero element in this line.
\end{proof}
\begin{rem}
The observation that Hard Lefshetz condition imposes strong constraints on the exceptional locus of the contraction is due to Wisniewski \cite{WisLef}.
\end{rem}
Now we are ready to prove that $Y$ is also rationally connected. If $Y$ is not rationally connected, then $Y$ has a unique $K_Y$-negative extremal ray (otherwise $Y$ is Fano and thus rationally connected). And the contraction $Y \to B$ is a fiber type contraction with fiber dimension $2$ by Lemma \ref{lem:finestructure}. By Lemma \ref{lem:extremalray}, there is also a contraction $X \to \Sigma$ which contracts the same extremal ray and has fiber dimension $2$. By Theorem 1.16 in \cite{SymGHS}, there is a curve class $\beta$, not in the extremal ray, such that $\langle [pt], \ldots \rangle^X_{0, \beta}\neq 0$. The same is true for $Y$. Thus $B$ is uniruled by the image of rational curves of class $\beta$, which is impossible since $B$ is the MRC quotient of $Y$ by Lemma \ref{lem:structure}.

In fact a little bit more is true.
\begin{thm}\label{thm:KollarConj}
Let $W$ and $W'$ be two smooth projective uniruled $4$-fold which are symplectic deformation equivalent and have second Betti number $2$. Then the MRC fibration of $W$ and $W'$ has the same dimension and the class of a general fiber has the same cohomology class.
\end{thm}
\begin{proof}
If one of them is rationally connected, then the theorem follows from the above discussion. Otherwise either both $W$ and $W'$ has a contraction which realizes the MRC quotient or the MRC quotient is a $3$-fold. In any case, there is a rational map $W \dashrightarrow S$ which is a proper morphism when restricted to a Zariski open subset and a general fiber is a smooth Fano variety $F$. To see the general fiber has the same dimension and cohomology class, we use the fact that for any Fano manifold $F$ of dimension at most $3$, there is a non-zero Gromov-Witten invariant of the form
\[
\langle [pt], [pt], \gamma_1, \ldots, \gamma_n \rangle^F_{0, \beta}.
\]
Furthermore, one can choose the classes $\gamma_i$ to be intersections of very ample divisors. In dimension $1$ and $2$, this is easy to see by classification. In dimension $3$, this is proved in \cite{SymRC}, Theorem 4.1 and Corollary 5.11. It follows from the proof in dimension $3$ (and easy to check in dimension $1$ and $2$) that any curve with the same $-K_F$ degree that can meet a general choice of the constraints are irreducible.

Choose cohomology classes $\Gamma_1, \ldots, \Gamma_n$ in $H^*(W, \QQ)$ such that $\Gamma_i \vert_F=\gamma_i$ for all $i$.
Now consider the linear map:
\[
H^*(W, \QQ) \to \QQ
\]
\[
 \alpha \mapsto \langle [pt], \alpha, \Gamma_1,\ldots, \Gamma_n \rangle^W_{0, i_*\beta}\\
\]
Clearly this map is invariant under symplectic deformations. But on the other hand, this is the same as the linear map given by
\[
\alpha \mapsto (\alpha \cdot [F]) (\sum_{\delta, i_*(\delta)=i_*(\beta)}\langle [pt], [pt], \gamma_1, \ldots, \gamma_n \rangle^F_{0, \delta}).
\]
Note that the summation term is positive and only depends on $F$. Thus one can read the information of the cohomolgy class of a general fiber from this linear map.
\end{proof}

\section{Proof of Theorem \ref{thm:ray}}
Let $R$ be an extremal ray of $Z$. If the contraction is a fiber type contraction, then there is a non-zero Gromov-Witten invariant of the form $\langle [pt], \ldots, \rangle^Z_{0, \beta}$ for some curve class $\beta$ in $R$. So from now on assume the contraction is birational. Let $C$ be a rational curve whose curve class $[C]$ is the minimal class in $R$ which can be represented by a rational curve. Finally, let $Z \to Y$ be the contraction and $L$ the pull back of a ample divisor on $Y$. Consider the linear maps:
\[
L^k : H^{n-k}(Z, \QQ) \to H^{n+k}(Z, \QQ)
\]
\[
\alpha \mapsto c_1(L)^k \cup \alpha,
\]
where $n=\dim Z$ and $1\leq k \leq n-2$.

The following lemma essentially comes from \cite{DefNefValue}, \cite{WisLef}.
\begin{lem}\label{lem:HL}
If all the linear maps $L^k (1 \leq k \leq n-2)$ are isomorphisms, then there is a non-zero Gromov-Witten invariant of the form $\langle \ldots \rangle^Z_{0, [C]}$.
\end{lem}
\begin{proof}

 Let $\overline{\mathcal{M}}$ be an irreducible component of $\overline{\mathcal{M}}_{0, 2}(Z, [C])$.
Consider the evaluation map:
 \[
 ev: \overline{\mathcal{M}} \to Z \times Z
 \]
  Then the fiber over a point not in the diagonal of $Z \times Z$ is finite. Otherwise one can deform the curve with two points fixed and by bend-and-break, the curve will degenerate to a reducible curve or a multiple cover of some curve, which is impossible by the minimality of $[C]$. Let $E\subset X$ be the image of the above map composed with the projection onto the first factor, thus an irreducible subvariety of the exceptional locus of $Z \to Y$.  Let $d$ (resp. $f$) be the dimension of the image of $E$ in $Y$ (resp. a general fiber). Clearly $d+f \leq n-1$. And the image of the evaluation map is at most $d+2f$.
 On the other hand, we have
 \[
 \dim \overline{\mathcal{M}} \geq -K_Z \cdot C+n-3+2 \geq n.
 \]
  Since the evaluation map is generically finite, $d+2f \geq n$. Then $d+f \geq n/2$ and equality holds if and only if $d=0$ and $2f=n$, in particular $d+2f=n$.

  Now assume equality does not hold. Notice that $c_1(L)^{d+1} \cup [E]=0$, and $[E] \in H^{n-(2d+2f-n)}(Z, \QQ)$. Note that $1 \leq d+1 \leq 2d+2f-n \leq n-2$ since $d+f \leq n-1$. Thus by assumption, $L^{2d+2f-n}$ is injective. Then $d\geq n-2$ otherwise $L^{d+1}$ is injective, a contradiction to the fact that $c_1(L)^{d+1} \cdot [E]=0$. But $d+f\leq n-1$ and $f \geq 1$. So $d=n-2, f=1$ and $d+2f=n$.

  In any case, we have the equality $d+2f=n$ and every irreducible component of $\overline{\mathcal{M}}_{0, 2}(X, [C])$ has the expected dimension $-K_Z \cdot C+n-1=n$. So is every irreducible component of $\overline{\mathcal{M}}_{0, m}(Z, [C])$. Thus the Gromov-Witten invariants are enumerative in the sense that they count the number of curves in class $[C]$ satisfying certain incidence constraints. In particular, some Gromov-Witten invariant (e.g. $\langle [H]^2, \ldots, [H]^2 \rangle^Z_{0, [C]}$, where $H$ is a very ample divisor) is non-zero.
\end{proof}

By the above lemma, we may assume some $L^k$ is not an isomorphism. But if the extremal ray $R$ is not effective, then the effective cone lies on one of the half spaces $\{\alpha \in H_2(X, \QQ) \vert \alpha \cdot L >0\}$ or  $\{\alpha \in H_2(X, \QQ) \vert \alpha \cdot L <0\}$. Here we need the assumption that there is an ample class $A \in H^2(Z, \QQ)=H^2(Z', \QQ)$. So either $L$ or $-L$ is an ample class in $H^2(Z', \QQ)$. Thus $L^k$ is an isomorphism for all $1 \leq k \leq n-2$ by the Hard Lefschetz Theorem. Therefore all extremal rays stay effective under symplectic deformations of algebraic varieties with second Betti number $2$. It is easy to see that at least one of the extremal rays is still extremal.

\bibliographystyle{alpha}
\bibliography{MyBib}

\end{document}